\documentclass[11pt,a4paper,draft]{article}
\usepackage{a4}

\usepackage{mathrsfs} %mathscr
\usepackage{amssymb} %Box
\usepackage{amsmath} %Operator
\usepackage{amsthm} %proof
\usepackage{stmaryrd} %bracket
\usepackage{xcolor}
\usepackage{graphicx}
\usepackage{hyperref}

\usepackage{tikz}
\usetikzlibrary{positioning}

\newtheorem{theorem}{Theorem}
\newtheorem{proposition}[theorem]{Proposition}

\newtheorem{conjecture}[theorem]{Conjecture}
\newtheorem{claim}[theorem]{Claim}
\newenvironment{claimproof}[1]{\par\noindent\emph{Proof#1.}\space}{\hfill $\Diamond$\medskip}

\newcommand{\bigo}[1]{O\left(#1\right)}

\DeclareMathOperator{\girth}{girth}

\title{Colouring powers and girth}

\author{
 Ross J. Kang
\thanks{Radboud University Nijmegen, Netherlands. Email: {\tt ross.kang@gmail.com}.}
 \and
 Fran\c{c}ois Pirot
\thanks{Ecole Normale Sup\'erieure de Lyon, France. Email: {\tt francois.pirot@ens-lyon.fr}.}
}

\begin{document}
\maketitle

\begin{abstract}
Alon and Mohar (2002) posed the following problem: among all graphs $G$ of maximum degree at most $d$ and girth at least $g$, what is the largest possible value of $\chi(G^t)$, the chromatic number of the $t$\textsuperscript{th} power of $G$? For $t\ge 3$, we provide several upper and lower bounds concerning this problem, all of which are sharp up to a constant factor as $d\to \infty$.
The upper bounds rely in part on the probabilistic method, while the lower bounds are various direct constructions whose building blocks are incidence structures.

%  Keywords: graph colouring, distance colouring, girth, graph powers, incidence structures.
  
%  AMS 2010 codes: 
%  05C15 (primary), %Coloring of graphs and hypergraphs
%  05C35, %Extremal problems
%  05C70 (secondary). %Factorization, matching, partitioning, covering and packing
\end{abstract}

%%%%%%%%%%%%%%%%%%%%%%%%%%%%%%%%%%%%%%%%%%%%%%%%%%%%%%%%%%%%%%%%%%%%%%

\section{Introduction}\label{sec:intro}

For a positive integer $t$, the {\em $t$-th power} $G^t$ of a (simple) graph $G = (V,E)$ is a graph with vertex set $V$ in which two distinct elements of $V$ are joined by an edge if there is a path in $G$ of length at most $t$ between them.
What is the largest possible value of the chromatic number $\chi(G^t)$ of $G^t$, among all graphs $G$ with maximum degree at most $d$ and girth (the length of the shortest cycle contained in the graph) at least $g$?

For $t=1$, this question was essentially a long-standing problem of Vizing~\cite{Viz68}, one that stimulated much work on the chromatic number of bounded degree triangle-free graphs, and was eventually settled asymptotically by Johansson~\cite{Joh96} using the probabilistic method. In particular, he showed that the largest possible value of the chromatic number over all girth $4$ graphs of maximum degree at most $d$ is $\Theta(d/\log d)$ as $d\to \infty$.

The case $t=2$ was considered and settled asymptotically by Alon and Mohar~\cite{AlMo02}. They showed that the largest possible value of the chromatic number of a graph's square taken over all girth $7$ graphs of maximum degree at most $d$ is $\Theta(d^2/\log d)$ as $d\to \infty$. Moreover, there exist girth $6$ graphs of arbitrarily large maximum degree $d$ such that the chromatic number of their square is $(1+o(1))d^2$ as $d\to\infty$.

In this work, we consider this extremal question for larger powers $t\ge 3$, which was posed as a problem in~\cite{AlMo02}, and settle a range of cases for $g$.

A first basic remark to make is that, ignoring the girth constraint, the maximum degree $\Delta(G^t)$ of $G^t$ for $G$ a graph of maximum degree at most $d$ satisfies
\begin{align*}
\Delta(G^t)
\le d\sum_{i=1}^{t-1}(d-1)^i
\le d^t,
\end{align*}
and therefore we have the following as a trivial upper bound for our problem:
\begin{align}\label{eqn:basic}
\chi(G^t) \le \Delta(G^t) + 1 \le d^t + 1.
\end{align}
This bound is sharp up to a $1+o(1)$ factor (as $d\to\infty$) for $t = 1$ and $g=3$, for $t=2$ and $g\le 6$ (as $d\to\infty$), but only two other cases for $t$ and $g$ have been settled to this precision, by examples for the so-called degree diameter problem, cf.~\cite{MiSi13}.
Recall that the De Bruijn graph of dimension $n$ on an alphabet $\Sigma$ of size $k$ is the directed graph whose vertices are the words of $\Sigma^n$ and whose arcs link the pairs $(a.u, u.b)$ for all $a,b \in \Sigma$, $u \in \Sigma^{n-1}$. For all $d$ even, the undirected and loopless version of the De Bruijn graph of dimension $t$ on an alphabet of size $d/2$ contains $d^t/2^t$ vertices, and every pair of its vertices can be linked with a path of length at most $t$; this certifies the general upper bound~\eqref{eqn:basic} to be sharp only up to a $(1+o(1))2^t$ factor.
It is known, via the degree diameter problem, that this factor can be improved upon in many cases for $t$.
However, De Bruijn graphs (and other constructions) have many short cycles and we are mostly interested here in whether the bound in~\eqref{eqn:basic} can be attained up to a constant factor by, or instead significantly lowered for, those graphs $G$ having some prescribed girth.

Alon and Mohar showed that the largest possible value of the chromatic number $\chi(G^t)$ of $G^t$, among all graphs $G$ with maximum degree at most $d$ and girth at least $3t+1$ is $\Theta(d^t/\log d)$ as $d\to \infty$. Kaiser and the first author~\cite{KaKa14} remarked that the same statement with $3t+1$ replaced by $2t+3$ could hold. In our first result, we make a further improvement by proving it necessary to exclude only the cycles of length $6$ if $t=2$ or of length in $\{8,10,\dots,2t+2\}$ when $t\ge 3$ in order to obtain an asymptotic reduction upon the bound in~\eqref{eqn:basic}.

\begin{theorem}\label{thm:cyclefree}
The largest possible value of the chromatic number $\chi(G^t)$ of $G^t$, taken over all graphs $G$ of maximum degree at most $d$ containing as a subgraph no cycle of length $6$ for $t=2$ or of length in $\{8,10,\dots,2t+2\}$ for $t\ge 3$ is $\Theta(d^t/\log d)$ as $d\to \infty$.
\end{theorem}

For $t=3$, this says that the chromatic number of the cube of a graph of maximum degree at most $d$ containing no cycle of length $8$ is $\Theta(d^3/\log d)$ as $d\to \infty$.
The largest forbidden cycle length $2t+2$ in Theorem~\ref{thm:cyclefree} may not in general be reduced to $2t+1$ or $2t$ because of the girth $6$ examples mentioned for case $t=2$. 
We made no effort to optimise the constant factors implicit in the $\Theta(d^t/\log d)$ term of Theorem~\ref{thm:cyclefree}, although doing so could be of interest in, say, the $t=2$ and $t=3$ cases.
We prove Theorem~\ref{thm:cyclefree} in Section~\ref{sec:cyclefree}.

We make a side remark that, with respect to the case $t=1$, excluding {\em any} fixed cycle length is sufficient for a logarithmic improvement over~\eqref{eqn:basic}.
\begin{proposition}\label{prop:cyclefree}
Let $k\ge 3$.
The largest possible value of the chromatic number $\chi(G)$ of $G$, taken over all graphs $G$ of maximum degree at most $d$ containing as a subgraph no cycle of length $k$ is $\Theta(d/\log d)$ as $d\to \infty$.
\end{proposition}

Before continuing, we introduce some abbreviating notation:
\begin{align*}
\chi^t_g(d) := \max\{\chi(G^t) : \Delta(G) \le d\text{ and }g(G) \ge g\},
\end{align*}
where $\Delta(G)$ is the maximum degree of the graph $G$ and $g(G)$ is its girth.
In this language, we have $\chi^1_3(d) \sim d$, while $\chi^1_4(d) = \Theta(d/\log d)$, and $\chi^2_6(d) \sim d^2$, while $\chi^2_7(d) = \Theta(d^2/\log d)$.
Alon and Mohar showed $\chi^t_g(d) = \Omega(d^t/\log d)$ for every $t$ and $g$, and Theorem~\ref{thm:cyclefree} thus implies that $\chi^t_{2t+3}(d) = \Theta(d^t/\log d)$ as $d\to\infty$.

A motivating conjecture for us is one of Alon and Mohar, asserting that for every $t$ there is a critical girth $g_t$ such that $\chi^t_{g_t}(d) = \Theta(d^t)$ and $\chi^t_{g_t+1}(d) = \Theta(d^t/\log d)$, just as for $t=1$ ($g_1 = 3$) and $t=2$ ($g_2 = 6$). 
We are not aware of any previous work, for any $t\ge 3$, showing that $g_t$, if it exists, is greater than $3$.

Our second contribution in this work is to give $\Omega(d^t)$ lower bounds on $\chi^t_g(d)$ for various choices of $t$ and $g (\le 2t+2)$. We show, in particular, that $g_t$, if it exists, is at least $4$ for $t=3$, at least $6$ for all $t \ge 4$, and at least $8$ for all $t\ge 11$.

\begin{theorem}\label{thm:constructions}
There are constructions to certify the following statements hold.
\begin{enumerate}
\item
$\chi^3_4(d) \gtrsim d^3/2^3$ as $d\to\infty$ and $\chi^3_3(d) \gtrsim 3d^3/2^3$ for infinitely many $d$;
\item
$\chi^4_6(d) \gtrsim d^4/2^4$ as $d\to\infty$ and $\chi^4_4(d) \gtrsim 2d^4/2^4$ as $d\to\infty$;
\item
$\chi^5_6(d) \gtrsim d^5/2^5$ as $d\to\infty$ and $\chi^5_4(d) \gtrsim 5d^5/2^5$ for infinitely many $d$;
\item
$\chi^6_6(d) \gtrsim d^6/2^6$ as $d\to\infty$ and $\chi^6_6(d) \gtrsim 3d^6/2^6$ for infinitely many $d$;
\item
$\chi^7_6(d) \gtrsim 2d^7/2^7$ as $d\to\infty$;
\item
$\chi^8_6(d) \gtrsim d^8/2^8$ as $d\to\infty$ and $\chi^8_6(d) \gtrsim 3d^8/2^8$ for infinitely many $d$;
\item
$\chi^{10}_6(d) \gtrsim d^{10}/2^{10}$ as $d\to\infty$ and $\chi^{10}_6(d) \gtrsim 5d^{10}/2^{10}$ for infinitely many $d$;
\item
for $t = 9$ or $t\ge 11$, $\chi^t_8(d) \gtrsim d^t/2^t$ as $d\to\infty$, $\chi^t_8(d) \gtrsim 3d^t/2^t$ for infinitely many $d$,
and, if $5 | t$, then $\chi^t_8(d) \gtrsim 5d^t/2^t$ for infinitely many $d$.
\end{enumerate}
Moreover, these constructions are bipartite if $t$ is even.
\end{theorem}

These lower bounds are obtained by a few different direct methods, including a circular construction (Section~\ref{sec:circular}) and two other somewhat ad hoc methods (Section~\ref{sec:cliques}).

A summary of current known bounds for Alon and Mohar's problem is given in Table~\ref{tab:summary}. 
When reflecting upon the gaps between entries in the upper and lower rows, one should keep in mind that among graphs $G$ of maximum degree at most $d$ and of girth lying strictly within these gaps, the current best upper and lower bounds on the extremal value of $\chi(G^t)$ are off by only a $\log d$ factor from one another.
We would be intrigued to learn of any constructions that certify $\liminf_{t\to\infty}g_t = \infty$, or of any upper bound on $\limsup_{t\to\infty}(g_t+1)/t$ strictly less than $2$.

\begin{table}
\centering
\setlength\tabcolsep{5pt}
\begin{tabular}{| r | c c c c c c c c c c c |}
\hline
$t$          & 1 & 2 & 3 & 4  & 5  & 6  & 7  & 8  & 9  & 10 & $\ge 11$\\
\hline
$g_t \ge$    & 3 & 6 & 4 & 6  & 6  & 6  & 6  & 6  & 8  & 6  & 8 \\
\hline
$g_t+1 \le $ & 4 & 7 & 9 & 11 & 13 & 15 & 17 & 19 & 21 & 23 & $2t+3$\\
\hline
\end{tabular}
\caption{Bounds on the conjectured critical girth $g_t$ (if it exists). \label{tab:summary}}
\end{table}

\section{An upper bound for graphs without certain cycles}\label{sec:cyclefree}

The proof of Theorem~\ref{thm:cyclefree} relies on the following result due to Alon, Krivelevich and Sudakov~\cite{AKS99}, showing an upper bound on the chromatic number of a graph whose maximum neighbourhood density is bounded. This result invokes Johannson's result for triangle-free graphs and is thereby reliant on the probabilistic method.

\begin{theorem}[\cite{AKS99}]\label{thm:AKSsparse}
For all graphs $\hat{G} = (\hat{V},\hat{E})$ with maximum degree at most $\hat{\Delta}$ such that for each $\hat{v}\in \hat{V}$ there are at most $\frac{\hat{\Delta}^2}{f}$ edges spanning $N(\hat{v})$, it holds that $\chi(\hat{G}) = \bigo{\frac{\hat{\Delta}}{\log f}}$ as $\hat{\Delta}\to\infty$.
\end{theorem}

Before proving Theorem~\ref{thm:cyclefree}, let us warm up with a proof of Proposition~\ref{prop:cyclefree}.

\begin{proof}[Proof of Proposition~\ref{prop:cyclefree}]
Let $G$ be a graph of maximum degree at most $d$ with no cycle of length $k$. Let $x$ be any vertex of $G$ and consider the subgraph $G[N(x)]$ induced by the neighbourhood of $x$. It clearly has at most $d$ vertices. Since $G$ contains no cycle of length $k$, $G[N(x)]$ contains no path of length $k-2$. Thus, by a result of Erd\H{o}s and Gallai~\cite{ErGa59} on the Tur\'an number of paths, $G[N(x)]$ contains at most $(k-3)d/2$ edges. By applying Theorem~\ref{thm:AKSsparse} with $\hat{\Delta}=d$ and $f = 2d/(k-3)$, it follows that $\chi(G) = \bigo{d/\log d}$ as $d\to\infty$.  There are standard probabilistic examples having arbitrarily large girth that show this bound to be sharp up to a constant factor, cf.~\cite[Ex.~12.7]{MoRe02}.
\end{proof}

\begin{proof}[Proof of Theorem~\ref{thm:cyclefree}]
Alon and Mohar~\cite{AlMo02} showed that $\chi^t_g(d) = \Omega(d^t/\log d)$ as $d\to\infty$, so it suffices to provide the upper bounds.

Let $G$ be a graph of maximum degree at most $d$ satisfying the required forbidden cycle conditions.
Our plan is to apply Theorem~\ref{thm:AKSsparse} with $\hat{G} = G^t$, $\hat{\Delta} = d^t$ and $f=\Omega(d)$, directly obtaining the bound on $\chi(G^t)$ we desire.
For the proof of Theorem~\ref{thm:cyclefree}, it remains to show that the number of edges spanning the neighbourhood of any vertex in $G^t$ is $O(d^{2t-1})$.

Let $x$ be any vertex of $G$. Let us denote by $A_i = A_i(x)$ the set of vertices of $G$ at distance exactly $i$ from $x$. Clearly, we have that $|A_i|\le\sum_{j=1}^i|A_j|\le d^i$ for all $i$.
We want to show that the number of pairs of vertices from $\bigcup_{i=1}^t A_i$ that are within distance $t$ of one another is $O(d^{2t-1})$. Note that the number of such pairs with one vertex in $A_i$, for some $i < t$, is at most $d^{t+i} \le d^{2t-1}$.
In fact, we shall show the stronger assertion that the number of paths of length $t$ both of whose endpoints are within distance $t$ of $x$ is $O(d^{2t-1})$. Note here that the number of such paths of length $i$, for some $i<t$, is at most $d^t(d-1)^i < d^{2t-1}$.

With the assumption on forbidden cycle lengths in $G$, we show the following claims.
\begin{claim}\label{clm:cyclefree1}
The induced subgraph $G[A_{t-1}\cup A_t]$ contains no $6$-path $x_1y_1x_2y_2x_3y_3x_4$ such that $x_i \in A_{t-1}$ for all $i$.
\end{claim}
\begin{claim}\label{clm:cyclefree2}
The induced subgraph $G[A_t\cup A_{t+1}]$ contains no $6$-path $x_1y_1x_2y_2x_3y_3x_4$ such that $x_i \in A_t$ for all $i$.
\end{claim}

\begin{claimproof}{ of Claim~\ref{clm:cyclefree1}}
For $t=2$, $G[A_1\cup A_2]$ contains no $4$-path $x_1y_1x_2y_2x_3$ such that $x_1,x_2,x_3\in A_1$, or else this path together with $x$ forms a cycle of length $6$.

So suppose $t\ge 3$
and $x_1y_1x_2y_2x_3y_3x_4$ is a $6$-path in $G[A_{t-1}\cup A_t]$ such that $x_i \in A_{t-1}$ for all $i$. In a breadth-first search tree rooted at $x$, let $a$ be the last common ancestor of $x_1$ and $x_3$. If $a \in A_i$ for some $i < t-2$, then there is a cycle containing the vertices $a$, $x_1$, $y_1$, $x_2$, $y_2$, $x_3$ of length $2(t-1-i)+4 \in \{8,10,\dots,2t+2\}$. So $a \in A_{t-2}$. Similarly, if $b$ is the last common ancestor of $x_2$ and $x_4$, then $b\in A_{t-2}$. If $a=b$, then $ax_1y_1x_2y_2x_3y_3x_4a$ is a cycle of length $8$. Otherwise, $ax_1y_1x_2bx_4y_3x_3a$  is a cycle of length $8$. In all cases we obtain a contradiction to the forbidden cycle conditions.
\end{claimproof}
\begin{claimproof}{ of Claim~\ref{clm:cyclefree2}}
For $t=2$, suppose that $G[A_2\cup A_3]$ contains a $4$-path $x_1y_1x_2y_2x_3$ such that $x_1,x_2,x_3\in A_2$. In a breadth-first search tree rooted at $x$, let $a$ be the last common ancestor of $x_1$ and $x_3$. It must be that $a=x$ or else $a$ together with the $4$-path forms a cycle of length $6$.
Then we consider the last common ancestors $a'$ of $x_1$ and $x_2$ and $a''$ of $x_2$ and $x_3$.
If neither of these is $x$, then $xa'x_1y_1x_2a''$ is a cycle of length $6$, a contradiction.
Otherwise, if say $a'$ is $x$, then there is a cycle containing $x$, $x_1$, $y_1$, $x_2$ that has length $6$, a contradiction.

So suppose $t\ge 3$
and $x_1y_1x_2y_2x_3y_3x_4$ is a $6$-path in $G[A_t\cup A_{t+1}]$ such that $x_i \in A_t$ for all $i$. We may attempt the same argument as for Claim~\ref{clm:cyclefree1}, except that the last common ancestor $a$ of $x_1$ and $x_3$ may well be in $A_0$, i.e.~$a=x$, which does not contradict the cycle condition. In this case, we consider the last common ancestors $a'$ of $x_1$ and $x_2$ and $a''$ of $x_2$ and $x_3$. 
If neither of these is $x$, then there is a cycle of length $2t+2$ which contains $x$, $a'$, $x_1y_1x_2$ and $a''$, a contradiction.
Otherwise, if say $a'$ is $x$, then there is a cycle containing $x$, $x_1$, $y_1$, $x_2$ that has length $2t+2$, a contradiction. So continuing as in the proof of Claim~\ref{clm:cyclefree1}, we obtain a contradiction in all cases.
\end{claimproof}

We can proceed with counting the (ordered) $t$-paths that join two vertices within distance $t$ of $x$. We use the two claims to estimate the number of $t$-paths containing particular vertices and edges.

Let us call a vertex $u$ of $G$ a {\em bottleneck of type~1} if $u\in A_t$ and $u$ has at least four neighbours in $A_{t-1}$.
We shall show that the vertices of $A_{t-1}$ are adjacent to at most $2$ bottlenecks of type~1 on average. Indeed, let $u'\in A_{t-1}$ be adjacent to $k\ge 3$ such bottlenecks, call them $y_1, y_2,\dots,y_k$. Let $Y = \{y_1,\dots,y_k\}$ and $U = N(Y)\cap A_{t-1} \setminus \{u'\}$. Every vertex of $U$ is adjacent to exactly one bottleneck of type~1. For otherwise suppose $u''\in U$ were adjacent to two bottlenecks of type~1, $y\in Y$ and $z$. Then, using $k\ge 3$ and the definition of a bottleneck of type~1, there would exist $a\in N(z) \cap A_{t-1}\setminus\{u',u''\}$, $y'\in Y\setminus\{y,z\}$ and $b\in N(y')\cap A_{t-1}\setminus\{a,u',u''\}$. And so $azu''yu'y'b$ would be a $6$-path in $G[A_{t-1}\cup A_t]$, a contradiction to Claim~\ref{clm:cyclefree1}. Since any bottleneck of type~1 has at least four neighbours in $A_{t-1}$, it follows that $|U|\ge 3k$. We then compute that the average adjacency to bottlenecks of type~1 from $\{u'\} \cup U$ is $(k+|U|)/(1+|U|) \le 1+(k-1)/(3k+1)<2$. Since this accounts for all vertices in $A_{t-1}$ adjacent to at least three bottlenecks of type~1, we conclude that the overall average adjacency from $A_{t-1}$ is also at most $2$, i.e.~the number of bottlenecks of type~1 is at most $2|A_{t-1}|<2d^{t-1}$. 
The number of $t$-paths containing a given vertex is at most $(t+1)d(d-1)^{t-1}$ (where $t+1$ counts its position within the path); therefore, the number of $t$-paths containing a bottleneck of type~1 is at most $2(t+1)d^t(d-1)^{t-1} < 3td^{2t-1}$.

Let us call an edge $uv$ of $G$ a {\em bottleneck of type~2} if $u\in A_t$, $v\in A_{t+1}$ and $v$ has at least four neighbours in $A_t$.
We shall show that the vertices of $A_t$ are incident to at most $2$ bottlenecks of type~2 on average. Indeed, let $u'\in A_t$ be incident to $k\ge 3$ such bottlenecks, call them $u'y_1, u'y_2,\dots,u'y_k$. Let $Y = \{y_1,\dots,y_k\}$ and $U = N(Y)\cap A_t \setminus \{u'\}$. Every vertex of $U$ is incident to exactly one bottleneck of type~2. For otherwise suppose $u''\in U$ were incident to two bottlenecks of type~2, $u''y$ with $y\in Y$ and $u''z$. Then, using $k\ge 3$ and the definition of a bottleneck of type~2, there would exist $a\in N(z) \cap A_t\setminus\{u',u''\}$, $y'\in Y\setminus\{y,z\}$ and $b\in N(y')\cap A_t\setminus\{a,u',u''\}$. And so $azu''yu'y'b$ would be a $6$-path in $G[A_t\cup A_{t+1}]$, a contradiction to Claim~\ref{clm:cyclefree2}. Since the $A_{t+1}$ endpoint of any bottleneck of type~2 has at least four neighbours in $A_t$, it follows that $|U|\ge 3k$. We then compute that the average incidence to bottlenecks of type~2 from $\{u'\} \cup U$ is $(k+|U|)/(1+|U|) \le 1+(k-1)/(3k+1)<2$. Since this accounts for all vertices in $\in A_t$ incident to at least three bottlenecks of type~2, we conclude that the overall average incidence from $A_t$ is also at most two, i.e.~the number of bottlenecks of type~2 is at most $2|A_t|<2d^t$. 
The number of $t$-paths containing a given edge is at most $t(d-1)^{t-1}$; therefore, the number of $t$-paths containing a bottleneck of type~2 is at most $2d^tt(d-1)^{t-1} < 2td^{2t-1}$.

Let us call an edge $uv$ of $G$ a {\em bottleneck of type~3} if $u,v\in A_t$.
By Claim~\ref{clm:cyclefree2}, $G[A_t]$ contains no $6$-path, and so by the result of Erd\H{o}s and Gallai~\cite{ErGa59} it has at most $2.5|A_t| < 2.5d^t$ edges. We can thus conclude that the number of $t$-paths containing a bottleneck of type~3 is at most $2.5d^tt(d-1)^{t-1}<3td^{2t-1}$.

Now we concentrate on $t$-paths having both endpoints in $A_t$ and containing no bottleneck of any type.
Let $x_0x_1\cdots x_t$ be such a $t$-path. So $x_0$ and $x_t$ are in $A_t$ and are not bottlenecks of type~1, and no $x_ix_{i+1}$ is a bottleneck of type~2 or~3. Since $x_0x_1$ is not a bottleneck of type~3, there are two possibilities: $x_1\in A_{t-1}$ or there is some $i\in\{1,\dots,t-1\}$ such that $x_i\in A_{t+1}$ and $x_{i+1}\in A_t$.
In the first case, since $x_0$ is not a bottleneck of type~1, there are at most three choices for $x_1$ and at most $(d-1)^{t-1}$ choices for the remainder of the path. In the second case, there are $t-1$ choices for $i$, there are at most three choices for $x_{i+1}$ given $x_i$ since $x_ix_{i+1}$ is not a bottleneck of type~2, and there are at most $(d-1)^{t-1}$ choices for the rest of the path.
Together with the at most $d^t$ choices for $x_0$, we combine the case considerations to conclude that the total number of choices for the path $x_0x_1\cdots x_t$ is at most $d^t(3+3(t-1))(d-1)^{t-1}<3td^{2t-1}$.

Overall, the number of pairs of vertices from $\bigcup_{i=1}^t A_i$ that are within distance $t$ of one another is at most $(2+11t)d^{2t-1}$. As $x$ was arbitrary, this implies that the number of edges spanning the neighbourhood of any vertex in $G^t$ is at most $(2+11t)d^{2t-1}$. An application of Theorem~\ref{thm:AKSsparse} to $G^t$ with $f = d/(2+11t)$ completes the proof.
\end{proof}

\section{Circular constructions}\label{sec:circular}

In this section, we describe some constructions based on a natural ``circular unfolding'' of the Hamming graph, or of the De Bruijn graph.
We first give basic versions that have weaker girth properties but provide intuition, and we develop these further later.

\begin{proposition}\label{prop:circularsimple}
For all positive $t$ and all even $d$, there is a graph $G$ of maximum degree $d$ such that $\chi(G^t) \ge d^t/2^t$. Moreover, $G$ can be chosen to have girth $4$ if $t\notin \{1,3\}$ as well as bipartite if $t$ is even.
\end{proposition}

\begin{proof}
Of course, the $t$-dimensional De Bruijn graph on $d/2$ symbols already certifies the first part of the statement, but we give two other constructions that satisfy the second part of the statement. Thus hereafter we can assume $t\ge 2$.
For both constructions, the vertex set is $V = \cup_{i=0}^{t-1}U^{(i)}$ where each $U^{(i)}$ is a copy of $[d/2]^t$, the set of ordered $t$-tuples of symbols from $[d/2]=\{1,\dots,d/2\}$.

\begin{description}
\item[A De Bruijn-type construction.]
We define $G_1 = (V,E_1)$ as follows.
For all $i\in\{0,\dots,t-1\}$, we join an element $(x^{(i)}_0,\dots,x^{(i)}_{t-1})$ of $U^{(i)}$ and an element $(x^{(i+1 \bmod t)}_0,\dots,x^{(i+1 \bmod t)}_{t-1})$ of $U^{(i+1 \bmod t)}$ by an edge if the latter is a left cyclic shift of the former,
i.e.~if $x^{(i+1 \bmod t)}_j = x^{(i)}_{j+1}$ for all $j\in\{0,\dots,t-2\}$ (and $x^{(i)}_0$, $x^{(i+1 \bmod t)}_{t-1}$ are arbitrary from $[d/2]$).
\item[A Hamming-type construction.]
We define $G_2 = (V,E_2)$ as follows.
For all $i\in\{0,\dots,t-1\}$, we join an element $(x^{(i)}_0,\dots,x^{(i)}_{t-1})$ of $U^{(i)}$ and an element $(x^{(i+1 \bmod t)}_0,\dots,x^{(i+1 \bmod t)}_{t-1})$ of $U^{(i+1 \bmod t)}$ by an edge if the $t$-tuples agree on all symbols except possibly at coordinate $i$,
i.e.~if $x^{(i+1 \bmod t)}_j = x^{(i)}_j$ for all $j\in\{0,\dots,t-1\}\setminus\{i\}$ (and $x^{(i)}_i$, $x^{(i+1 \bmod t)}_i$ are arbitrary from $[d/2]$).
\end{description}
See Figure~\ref{fig:circularsimple} for a schematic of $G_2$.

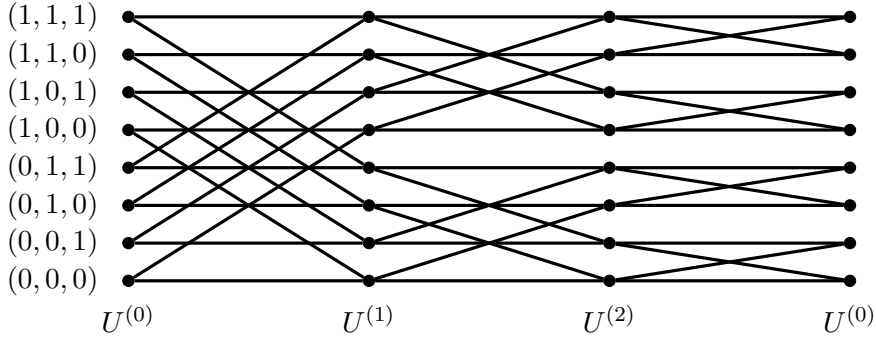
\begin{figure}
\centering
%!TEX root = distgirth.tex

\begin{tikzpicture}[-,>=,node distance=0.5cm,scale=1,draw,nodes={circle,draw,fill=black, inner sep=1.5pt}]

\node (a0) [left] {};
\node (a1) [above of=a0] {};
\node (a2) [above of=a1] {};
\node (a3) [above of=a2] {};
\node (a4) [above of=a3] {};
\node (a5) [above of=a4] {};
\node (a6) [above of=a5] {};
\node (a7) [above of=a6] {};

\node (z0) [left=0.2 of a0,fill=none,draw=none] {$(0,0,0)$};
\node (z1) [left=0.2 of a1,fill=none,draw=none] {$(0,0,1)$};
\node (z2) [left=0.2 of a2,fill=none,draw=none] {$(0,1,0)$};
\node (z3) [left=0.2 of a3,fill=none,draw=none] {$(0,1,1)$};
\node (z4) [left=0.2 of a4,fill=none,draw=none] {$(1,0,0)$};
\node (z5) [left=0.2 of a5,fill=none,draw=none] {$(1,0,1)$};
\node (z6) [left=0.2 of a6,fill=none,draw=none] {$(1,1,0)$};
\node (z7) [left=0.2 of a7,fill=none,draw=none] {$(1,1,1)$};

\node (b0) [right=3 of a0] {};
\node (b1) [above of=b0] {};
\node (b2) [above of=b1] {};
\node (b3) [above of=b2] {};
\node (b4) [above of=b3] {};
\node (b5) [above of=b4] {};
\node (b6) [above of=b5] {};
\node (b7) [above of=b6] {};

\node (c0) [right=3 of b0] {};
\node (c1) [above of=c0] {};
\node (c2) [above of=c1] {};
\node (c3) [above of=c2] {};
\node (c4) [above of=c3] {};
\node (c5) [above of=c4] {};
\node (c6) [above of=c5] {};
\node (c7) [above of=c6] {};

\node (d0) [right=3 of c0] {};
\node (d1) [above of=d0] {};
\node (d2) [above of=d1] {};
\node (d3) [above of=d2] {};
\node (d4) [above of=d3] {};
\node (d5) [above of=d4] {};
\node (d6) [above of=d5] {};
\node (d7) [above of=d6] {};

  \path (a0) edge [very thick] node[fill=none,draw=none] {} (b0);
  \path (a1) edge [very thick] node[fill=none,draw=none] {} (b1);
  \path (a2) edge [very thick] node[fill=none,draw=none] {} (b2);
  \path (a3) edge [very thick] node[fill=none,draw=none] {} (b3);
  \path (a4) edge [very thick] node[fill=none,draw=none] {} (b4);
  \path (a5) edge [very thick] node[fill=none,draw=none] {} (b5);
  \path (a6) edge [very thick] node[fill=none,draw=none] {} (b6);
  \path (a7) edge [very thick] node[fill=none,draw=none] {} (b7);
  \path (a0) edge [very thick] node[fill=none,draw=none] {} (b4);
  \path (a1) edge [very thick] node[fill=none,draw=none] {} (b5);
  \path (a2) edge [very thick] node[fill=none,draw=none] {} (b6);
  \path (a3) edge [very thick] node[fill=none,draw=none] {} (b7);
  \path (a4) edge [very thick] node[fill=none,draw=none] {} (b0);
  \path (a5) edge [very thick] node[fill=none,draw=none] {} (b1);
  \path (a6) edge [very thick] node[fill=none,draw=none] {} (b2);
  \path (a7) edge [very thick] node[fill=none,draw=none] {} (b3);

  \path (b0) edge [very thick] node[fill=none,draw=none] {} (c0);
  \path (b1) edge [very thick] node[fill=none,draw=none] {} (c1);
  \path (b2) edge [very thick] node[fill=none,draw=none] {} (c2);
  \path (b3) edge [very thick] node[fill=none,draw=none] {} (c3);
  \path (b4) edge [very thick] node[fill=none,draw=none] {} (c4);
  \path (b5) edge [very thick] node[fill=none,draw=none] {} (c5);
  \path (b6) edge [very thick] node[fill=none,draw=none] {} (c6);
  \path (b7) edge [very thick] node[fill=none,draw=none] {} (c7);
  \path (b0) edge [very thick] node[fill=none,draw=none] {} (c2);
  \path (b1) edge [very thick] node[fill=none,draw=none] {} (c3);
  \path (b2) edge [very thick] node[fill=none,draw=none] {} (c0);
  \path (b3) edge [very thick] node[fill=none,draw=none] {} (c1);
  \path (b4) edge [very thick] node[fill=none,draw=none] {} (c6);
  \path (b5) edge [very thick] node[fill=none,draw=none] {} (c7);
  \path (b6) edge [very thick] node[fill=none,draw=none] {} (c4);
  \path (b7) edge [very thick] node[fill=none,draw=none] {} (c5);

  \path (c0) edge [very thick] node[fill=none,draw=none] {} (d0);
  \path (c1) edge [very thick] node[fill=none,draw=none] {} (d1);
  \path (c2) edge [very thick] node[fill=none,draw=none] {} (d2);
  \path (c3) edge [very thick] node[fill=none,draw=none] {} (d3);
  \path (c4) edge [very thick] node[fill=none,draw=none] {} (d4);
  \path (c5) edge [very thick] node[fill=none,draw=none] {} (d5);
  \path (c6) edge [very thick] node[fill=none,draw=none] {} (d6);
  \path (c7) edge [very thick] node[fill=none,draw=none] {} (d7);
  \path (c0) edge [very thick] node[fill=none,draw=none] {} (d1);
  \path (c1) edge [very thick] node[fill=none,draw=none] {} (d0);
  \path (c2) edge [very thick] node[fill=none,draw=none] {} (d3);
  \path (c3) edge [very thick] node[fill=none,draw=none] {} (d2);
  \path (c4) edge [very thick] node[fill=none,draw=none] {} (d5);
  \path (c5) edge [very thick] node[fill=none,draw=none] {} (d4);
  \path (c6) edge [very thick] node[fill=none,draw=none] {} (d7);
  \path (c7) edge [very thick] node[fill=none,draw=none] {} (d6);

\node (a) [below of=a0,fill=none,draw=none] {$U^{(0)}$};
\node (b) [below of=b0,fill=none,draw=none] {$U^{(1)}$};
\node (c) [below of=c0,fill=none,draw=none] {$U^{(2)}$};
\node (d) [below of=d0,fill=none,draw=none] {$U^{(0)}$};

\end{tikzpicture}
\caption{A schematic for the Hamming-type circular construction $G_2$ for $t=3$ and $d=4$.\label{fig:circularsimple}}
\end{figure}

In both constructions, the maximum degree into $U^{(i+1 \bmod t)}$ from a vertex in $U^{(i)}$ is $d/2$ and the same is true from $U^{(i+1 \bmod t)}$ into $U^{(i)}$, so both constructions have maximum degree $d$ overall.
In both constructions, for any pair of elements in $U^{(0)}$ there is a path between them of length at most $t$, one that passes through every $U^{(i)}$, either by a sequence of $t$ cyclic shifts or a sequence of $t$ one-symbol changes. Therefore, the induced subgraphs ${G_1}^t[U^{(0)}]$ and ${G_2}^t[U^{(0)}]$ are both cliques, implying that $\chi({G_1}^t) \ge |U^{(0)}| = d^t/2^t$ and similarly for $G_2$.
As these constructions are composed of bipartite graphs connected in sequence around a cycle of length $t$, $G_1$ and $G_2$ are of girth $4$ if $t\ne 3$ and are bipartite if $t$ is even. This ends the proof.
\end{proof}

To proceed further with this circular construction to obtain one with higher girth, we certainly have to handle the (many) cycles of length $4$ that span only two consecutive parts $U^{(i)}$ and $U^{(i+1 \bmod t)}$. We do this essentially by substituting a subsegment $U^{(i)}, U^{(i+1 \bmod t)}, \dots, U^{(i+k \bmod t)}$ with a sparse bipartite structure having good distance properties.
Some of the most efficient such sparse structures arise from finite geometries, generalised polygons in particular. We base our substitution operation on these structures, and therefore find it convenient to encapsulate the properties most relevant to us in the following definition.

We say a balanced bipartite graph $H = (V=A\cup B, E)$ with parts $A$ and $B$, $|A|=|B|$, is a {\em good conduit (between $A$ and $B$) with parameters $(\tau, \Delta, \gamma,c)$} if it has girth $\gamma$, it is regular of degree $\Delta$, there is a path of length at most $\tau$ between any $a\in A$ and any $b\in B$, and moreover $|A|$ (and so also $|B|$) is of maximum possible order $\Theta(\Delta^\tau)$ such that $|A| \ge c\Delta^\tau$.

The following good conduits are useful in our constructions, because of their relatively high girth.
The balanced complete bipartite graph $K_{\Delta,\Delta}$ is a good conduit with parameters $(1, \Delta, 4,1)$.
Let $q$ be a prime power. 
The point-line incidence graph ${\cal Q}_q$ of a symplectic quadrangle with parameters $(q,q)$ is a good conduit with parameters $(3,q+1,8,1)$.
The point-line incidence graph  ${\cal H}_q$ of a split Cayley hexagon with parameters $(q,q)$ is a good conduit with parameters $(5,q+1,12,1)$.
We have intentionally made specific classical choices of generalised polygons here, cf.~\cite{PaTh84}, partly because we know they are defined for all prime powers $q$ and partly for symmetry considerations described later.
We remark that no generalised octagon with parameters $(q,q)$ exists and no generalised $n$-gons for any other even value of $n$ exist~\cite{FeHi64}.

We are now prepared to present the main construction of the section.
This is a generalisation of $G_2$.
(It is possible to generalise $G_1$ in a similar way.)

\begin{theorem}\label{thm:circularmain}
Let $t = \sum_{i=0}^{\lambda-1}\tau_i$ for some positive odd integers $\tau_i$ and $\lambda\ge2$. Let $d$ be even.
Suppose that for every $i$ there is a good conduit $H_i$ with parameters $(\tau_i,d/2,\gamma_i,c_i)$.
Then there is a graph $G$ of maximum degree $d$ such that $\chi(G^t) \ge (\prod_{i=0}^{\lambda-1}c_i) d^t/2^t$ and its girth satisfies
$\girth(G) \ge \min\{\lambda,8,\min_i \gamma_i\}$.
Moreover, $G$ is bipartite if and only if $t$ is even.
\end{theorem}

After the proof, we show how to modify the construction in certain cases to mimic the inclusion of good conduits with $\tau$ parameter equal to $2$ (note that good conduits with even $\tau$ are precluded from the definition), to increase the girth of $G$, or to improve the bound on $\chi(G^t)$.

\begin{proof}[Proof of Theorem~\ref{thm:circularmain}]
For every $i$, let $H_i = (V_i=A_i\cup B_i,E_i)$ be the assumed good conduit with parameters $(\tau_i,d/2,\gamma_i,c_i)$.
Write $A_i = \{a^{i}_1,\dots,a^{i}_{n_i}\}$ and $B_i = \{b^{i}_1,\dots,b^{i}_{n_i}\}$.
By the definition of $H_i$, $n_i \ge c_i d^{\tau_i}/2^{\tau_i}$.

We define $G = (V,E)$ as follows.
The vertex set is $V = \cup_{i=0}^{\lambda-1}U^{(i)}$ where each $U^{(i)}$ is a copy of $\prod_{j=0}^{\lambda-1}[n_j]$, the set of ordered $\lambda$-tuples whose $j$\textsuperscript{th} coordinate is a symbol from $[n_j]=\{1,\dots,n_j\}$.
For all $i\in\{0,\dots,\lambda-1\}$, we join an element $(x^{(i)}_0,\dots,x^{(i)}_{\lambda-1})$ of $U^{(i)}$ and an element $(x^{(i+1 \bmod \lambda)}_0,\dots,x^{(i+1 \bmod \lambda)}_{\lambda-1})$ of $U^{(i+1 \bmod \lambda)}$ by an edge only if the $\lambda$-tuples agree on all symbols except possibly at coordinate $i$, in which case we use $H_{i}$ and its ordering as a template for adjacency.
More precisely, join $(x^{(i)}_0,\dots,x^{(i)}_{\lambda-1})$ and $(x^{(i+1 \bmod \lambda)}_0,\dots,x^{(i+1 \bmod \lambda)}_{\lambda-1})$ by an edge of $G$ if $x^{(i)}_j = x^{(i+1 \bmod \lambda)}_j$ for all $j\in\{0,\dots,\lambda-1\}\setminus\{i\}$ and there is an edge in $H_{i}$ joining $a^{i}_{x^{(i)}_{i}}$ and $b^{i}_{x^{(i+1 \bmod \lambda)}_{i}}$.
Clearly $G$ is bipartite if and only if $\lambda$ is even, which holds if and only if $t$ is even.

Since $H_i$ has maximum degree $d/2$, the degree into $U^{(i+1 \bmod \lambda)}$ from a vertex in $U^{(i)}$ is at most $d/2$ (with respect to the edges added between $U^{(i)}$ and $U^{(i+1 \bmod \lambda)}$) and the same is true from $U^{(i+1 \bmod \lambda)}$ into $U^{(i)}$, so overall $G$ has maximum degree $d$.
Between any pair of elements in $U^{(0)}$, there is a path of length at most $t=\sum_{i=0}^{\lambda-1}\tau_i$ passing through every $U^{(i)}$, which changes the symbol at coordinate $i$ via a path of length at most $\tau_i$ in the subgraph induced by $U^{(i)}$ and $U^{(i+1 \bmod \lambda)}$.
It follows that the induced subgraph $G^t[U^{(0)}]$ is a clique, and so $\chi(G^t) \ge |U^{(0)}| = \prod_{i=0}^{\lambda-1}n_i \ge \prod_{i=0}^{\lambda-1}c_id^{\tau_i}/2^{\tau_i} = (\prod_{i=0}^{\lambda-1}c_i)d^t/2^t$.

All that remains is to establish the girth of $G$.
For the statement we essentially only need to consider cycles of length $7$ or less,
whose winding number with respect to the cycle $U^{(0)}U^{(1)}\cdots U^{(\lambda)}U^{(0)}$ is $0$.
Such cycles are of even length, so we only need to consider lengths $4$ and $6$.
We do not need to consider the cycles that only go back and forth between $U^{(i)}$ and $U^{(i+1 \bmod \lambda)}$ (only along the edges added between $U^{(i)}$ and $U^{(i+1 \bmod \lambda)}$), for such cycles are accounted for by the $\min_i \gamma_i$ term.
So, for cycles of length $4$ of winding number $0$, without loss of generality we need only consider one that proceeds in order through $U^{(0)}$, $U^{(1)}$, $U^{(2)}$, and then back through $U^{(1)}$ to $U^{(0)}$, written as $u^{(0)}u^{(1)}u^{(2)}v^{(1)}u^{(0)}$. By construction, the $\lambda$-tuples $u^{(0)}$, $u^{(1)}$, $v^{(1)}$ share all but their zeroth coordinate symbols and the tuples $u^{(2)}$, $u^{(1)}$, $v^{(1)}$ share all but their first coordinate; however, this implies that $u^{(1)}$ and $v^{(1)}$ are the same tuple in $U^{(1)}$, a contradiction.
For cycles of length $6$ that, say, proceed in order through $U^{(0)}$, $U^{(1)}$, $U^{(2)}$, $U^{(3)}$ and back, we argue in a similar fashion as for length $4$ to obtain a contradiction.
The remaining case (for winding number $0$) is a cycle of length $6$ that is, without loss of generality, of the form $u^{(0)}u^{(1)}u^{(2)}v^{(1)}v^{(2)}w^{(1)}u^{(0)}$. By construction, the tuples $u^{(0)}$, $u^{(1)}$, $w^{(1)}$ share all but their zeroth coordinate symbols and the tuples $v^{(2)}$, $v^{(1)}$, $w^{(1)}$ share all but their first coordinate as do $u^{(2)}$, $u^{(1)}$, $v^{(1)}$; however, this implies that $u^{(1)}$ and $w^{(1)}$ are the same tuple in $U^{(1)}$, a contradiction.
This concludes our determination of the girth of $G$.

We remark that cycles of length $8$ may well occur, for instance when the same good conduit $H$ is used two times consecutively.
In particular, supposing $H$ is used from $U^{(0)}$ to $U^{(1)}$ to $U^{(2)}$ and $a_1b_2a_3$ and $b_4a_5b_6$ are two $2$-paths in $H$, then
$(4,1,\ldots)^{(1)}$, $(4,2,\ldots)^{(2)}$, $(4,3,\ldots)^{(1)}$, $(5,3,\ldots)^{(0)}$, $(6,3,\ldots)^{(1)}$, $(6,2,\ldots)^{(2)}$, $(6,1,\ldots)^{(1)}$, $(5,1,\ldots)^{(0)}$, $(4,1,\ldots)^{(1)}$ represents an $8$-cycle in the construction.
\end{proof}

In two of the small values for $t$ (namely, $4$ or $7$), we cannot apply the construction of Theorem~\ref{thm:circularmain} without a modification. The intuition is to include another sparse structure with good distance properties, that is, the point-line incidence graph ${\cal P}_q$ of the projective plane $PG(2,q)$, for $q$ a prime power.
This is a bipartite graph $(V=A\cup B, E)$ of girth $6$, that is regular of degree $q+1$, has a $2$-path between $a$ and $a'$ for any $a,a'\in A$ (and similarly a $2$-path between $b$ and $b'$ for any $b,b'\in B$), and has $|A| = |B| = q^2+q+1$.
The graph ${\cal P}_{d-1}$ certifies $\chi^2_6(d) \gtrsim d^2$ if $d-1$ is a prime power; moreover, since the gap between two successive primes $p$ and $p'$ is $o(p)$~\cite{Ing37}, the inequality holds for all $d$ as $d\to\infty$.

The graph ${\cal P}_q$ has properties similar to what we might require for a good conduit having parameter $\tau=2$, except that it connects vertices in the same part.
One solution to this parity issue is to ``unfold a mirror of ${\cal P}_q$'', that is, add a disjoint copy of one of its parts with the same adjacencies as the original, so that the conduit is between the vertices of two copies of the same part.
See Figure~\ref{fig:mirror} for an illustration of ${\cal P}_2$ together with its mirror, denoted by $-{{\cal P}_2}$.

\begin{figure}
\centering
%!TEX root = distgirth.tex

\begin{tikzpicture}[-,>=,node distance=0.5cm,scale=1,draw,nodes={circle,draw,fill=black, inner sep=1.5pt}]

\node (a1) [] {};
\node (a2) [above of=a1] {};
\node (a3) [above of=a2] {};
\node (a4) [above of=a3] {};
\node (a5) [above of=a4] {};
\node (a6) [above of=a5] {};
\node (a7) [above of=a6] {};

\node (b1) [right=3 of a2] {};
\node (b2) [above of=b1] {};
\node (b3) [above of=b2] {};
\node (b4) [above of=b3] {};
\node (b5) [above of=b4] {};
\node (b6) [above of=b5] {};
\node (b7) [above of=b6] {};

\node (c2) [right=3 of b1] {};
\node (c1) [below of=c2] {};
\node (c3) [above of=c2] {};
\node (c4) [above of=c3] {};
\node (c5) [above of=c4] {};
\node (c6) [above of=c5] {};
\node (c7) [above of=c6] {};

  \path (a1) edge [very thick] node[fill=none,draw=none] {} (b2);
  \path (a1) edge [very thick] node[fill=none,draw=none] {} (b3);
  \path (a1) edge [very thick] node[fill=none,draw=none] {} (b4);
  \path (a2) edge [very thick] node[fill=none,draw=none] {} (b1);
  \path (a2) edge [very thick] node[fill=none,draw=none] {} (b2);
  \path (a2) edge [very thick] node[fill=none,draw=none] {} (b5);
  \path (a3) edge [very thick] node[fill=none,draw=none] {} (b1);
  \path (a3) edge [very thick] node[fill=none,draw=none] {} (b3);
  \path (a3) edge [very thick] node[fill=none,draw=none] {} (b6);
  \path (a4) edge [very thick] node[fill=none,draw=none] {} (b1);
  \path (a4) edge [very thick] node[fill=none,draw=none] {} (b4);
  \path (a4) edge [very thick] node[fill=none,draw=none] {} (b7);
  \path (a5) edge [very thick] node[fill=none,draw=none] {} (b2);
  \path (a5) edge [very thick] node[fill=none,draw=none] {} (b6);
  \path (a5) edge [very thick] node[fill=none,draw=none] {} (b7);
  \path (a6) edge [very thick] node[fill=none,draw=none] {} (b3);
  \path (a6) edge [very thick] node[fill=none,draw=none] {} (b5);
  \path (a6) edge [very thick] node[fill=none,draw=none] {} (b7);
  \path (a7) edge [very thick] node[fill=none,draw=none] {} (b4);
  \path (a7) edge [very thick] node[fill=none,draw=none] {} (b5);
  \path (a7) edge [very thick] node[fill=none,draw=none] {} (b6);

  \path (b1) edge [very thick] node[fill=none,draw=none] {} (c2);
  \path (b1) edge [very thick] node[fill=none,draw=none] {} (c3);
  \path (b1) edge [very thick] node[fill=none,draw=none] {} (c4);
  \path (b2) edge [very thick] node[fill=none,draw=none] {} (c1);
  \path (b2) edge [very thick] node[fill=none,draw=none] {} (c2);
  \path (b2) edge [very thick] node[fill=none,draw=none] {} (c5);
  \path (b3) edge [very thick] node[fill=none,draw=none] {} (c1);
  \path (b3) edge [very thick] node[fill=none,draw=none] {} (c3);
  \path (b3) edge [very thick] node[fill=none,draw=none] {} (c6);
  \path (b4) edge [very thick] node[fill=none,draw=none] {} (c1);
  \path (b4) edge [very thick] node[fill=none,draw=none] {} (c4);
  \path (b4) edge [very thick] node[fill=none,draw=none] {} (c7);
  \path (b5) edge [very thick] node[fill=none,draw=none] {} (c2);
  \path (b5) edge [very thick] node[fill=none,draw=none] {} (c6);
  \path (b5) edge [very thick] node[fill=none,draw=none] {} (c7);
  \path (b6) edge [very thick] node[fill=none,draw=none] {} (c3);
  \path (b6) edge [very thick] node[fill=none,draw=none] {} (c5);
  \path (b6) edge [very thick] node[fill=none,draw=none] {} (c7);
  \path (b7) edge [very thick] node[fill=none,draw=none] {} (c4);
  \path (b7) edge [very thick] node[fill=none,draw=none] {} (c5);
  \path (b7) edge [very thick] node[fill=none,draw=none] {} (c6);

\end{tikzpicture}
\caption{An illustration of ${\cal P}_2$ together with its mirror.\label{fig:mirror}}
\end{figure}
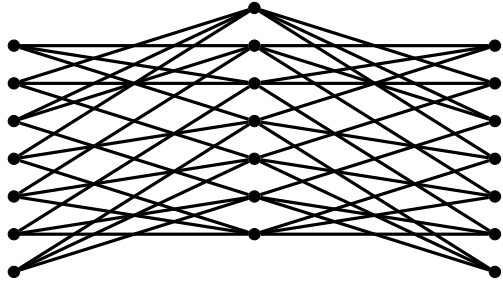

Directly, however, this creates cycles of length $4$ (from vertices of degree $2$), so we need to segregate the embedding of ${\cal P}_q$ and its mirror (i.e.~${\cal P}_q$ with $A$ and $B$ switched).
More precisely, suppose that we want to use ${\cal P}_q$ as a template for the edges between $U^{(0)}$ and $U^{(1)}$ (as in the construction of $G$ in Theorem~\ref{thm:circularmain}). For this, we change symbols (chosen from $[q^2+q+1]$) at the zeroth coordinate in one step according to ${\cal P}_q$ between $U^{(0)}$ and $U^{(1)}$, and change the zeroth coordinate in a second step according to the mirror of ${\cal P}_q$ but later in the cycle, say by adding a new part $U'^{(0)}$ after $U^{(2)}$ in the cycle and adding edges between $U^{(2)}$ and $U'^{(0)}$ according to the mirror of ${\cal P}_q$. Although this adds one more part to the cycle of $U^{(j)}$s, it avoids cycles in $G$ of length $4$ and appropriately mimics the distance properties of a good conduit with parameter $\tau=2$.
We can also interleave when we want to embed ${\cal P}_q$ and its mirror for two or three coordinates.

In all, the girth we obtain for this modification of Theorem~\ref{thm:circularmain} satisfies
$\girth(G) \ge \min\{\lambda+\iota,6,\min_i \gamma_i\}$, where $\iota = |\{i | \tau_i = 2\}| \ge 1$, provided that $\lambda \ge 3$ if $\iota = 1$.

\bigskip
For $t\ge 6$, it is possible to improve on the construction in Theorem~\ref{thm:circularmain} either in terms of the girth of $G$ or $\chi(G^t)$ by applying a similar modification as above but instead to good conduits.
In particular, we can ``unfold'' ${\cal Q}_q$ or ${\cal H}_q$ into three copies (one of which is mirrored), or possibly five copies (two of which are mirrored) in the case of ${\cal H}_q$, and distribute the embeddings of these copies around the cycle, doing this for all coordinates. By unfolding into an odd number ($\le \tau$) of parts, the distance properties of the construction are unhindered.
 If these embeddings are interleaved so that no two embeddings of the same coordinate are at distance at most $1$ in the cycle  $U^{(0)}U^{(1)}\cdots U^{(\lambda)}U^{(0)}$, then the same analysis for girth at the end of the proof of Theorem~\ref{thm:circularmain} applies. If they are interleaved so that no two embeddings of the same coordinate are at distance $0$, then cycles of length $4$ do not occur but cycles of length $6$ may well occur.

Furthermore, when all of the coordinates are unfolded into the same number (three or five) of copies and these are distributed evenly so that each segment of length $\lambda$ contains exactly one embedding for each coordinate, and the good conduits satisfy a symmetry condition (self-duality) that we describe formally in Section~\ref{sec:cliques}, then $U^{(0)} \cup U^{(\lambda)} \cup U^{(2\lambda)} \cup \cdots$ induces a clique in the $t$-th power, increasing the bound on $\chi(G^t)$ (by a factor $3$ or $5$).

\bigskip
The above modifications do not affect the parity of the main cycle, so we still have that the construction is bipartite if and only if $t$ is even.
We shall describe a few further special improvements upon Theorem~\ref{thm:circularmain} within the proof of Theorem~\ref{thm:constructions}.

\section{Cliques in $G^t$ from good conduits}\label{sec:cliques}

Good conduits of parameters $(\tau,\Delta,\gamma,c)$ are themselves very nearly cliques in the $\tau$\textsuperscript{th} power and indeed there are two simple ways to modify them to create such cliques. This yields better bounds for $\chi^t_g(d)$ in a few situations when $t\in\{3,5\}$.

The first idea is to contract a perfect matching, thereby merging the parts.

\begin{proposition}\label{prop:contraction}
Let $H$ be a good conduit of parameters $(\tau,\Delta,\gamma,c)$ and let $A = \{a_1,\dots,a_n\}$, $B = \{b_1,\dots,b_n\}$ be a matching ordering of $H$.
The graph $\mu(H)$, which we call the {\em matching contraction of $H$}, formed from $H$ by contracting every edge $a_ib_i$, $i\in\{1,\dots,n\}$ and ignoring any duplicate edges satisfies the following properties: $\mu(H)$ has maximum degree $2\Delta-2$, girth at least $\gamma/2$, $n (\ge c\Delta^\tau)$ vertices, and between every two vertices in $\mu(H)$ there is a path of length at most $\tau$.
\end{proposition}

\begin{proof}
The statements about the maximum degree and number of vertices are trivial to check.
That every pair of vertices is joined by a $\tau$-path follows from the distance properties of $H$ as a good conduit.
Let $\{v_1,\dots,v_n\}$ be an ordering of the vertices such that $v_i$ corresponds to the contracted edge $a_ib_i$ for every $i\in\{1,\dots,n\}$.
For the girth, suppose $C = v_{i_0}v_{i_1}\cdots v_{i_{\ell}}v_{i_0}$ is a cycle of length $\ell$ in $\mu(H)$.
Then, for every $j\in\{1,\dots,\ell\}$, either $a_{i_j}b_{i_{j+1\bmod \ell}}$ or $a_{i_{j+1\bmod \ell}}b_{i_j}$ is an edge of $H$.
Moreover, every $a_{i_j}b_{i_j}$ is an edge of $H$, so by also including at most $\ell$ such edges, we obtain a cycle of length at most $2\ell$ in $H$.
So $\girth(H) \le 2\girth(\mu(H))$, as required.
\end{proof}

\noindent
It is worth noting that the conclusion of the last proposition gives a lower bound on the output girth (of $\gamma/2$); however, it is conceivable that with a suitable matching ordering the girth could be made higher (namely, up to $3\gamma/4-1$), but we have not yet pursued this further.

The second idea is to connect good conduits with parameter $\tau$ end-to-end around a cycle of length $\tau$.
In order to do so, we need that the input good conduit is symmetric. More precisely, we say a good conduit between $A$ and $B$ is {\em self-dual} if there is a bijection $\sigma: A \to B$ such that the mapping for which every element $a \in A$ is mapped to $\sigma(a)$ and every element $b\in B$ is mapped to $\sigma^{-1}(b)$ is an automorphism of the graph.
In other words, a self-dual good conduit has an embedding such that it is isomorphic to its mirror.
Note that this corresponds to the notion of self-duality in generalised polygons, so every self-dual generalised polygon gives rise to a self-dual good conduit. It is known that ${\cal Q}_q$, resp.~${\cal H}_q$, is self-dual when $q$ is a power of $2$, resp.~of $3$, cf.~\cite{PaTh84}.

\begin{proposition}\label{prop:cycle}
Let $H$ be a self-dual good conduit of parameters $(\tau,\Delta,\gamma,c)$ for $\tau\ge 3$.
The graph $\psi(H)$, which we call an {\em $H$-cycle of length $\tau$}, formed by connecting $\tau$ copies of $H$ end-to-end in a cycle with vertices identified according to the self-duality bijection for $H$, is regular of degree $2\Delta$, has girth $\min\{\tau,4\}$, has $\tau n (\ge \tau c\Delta^\tau)$ vertices, and between every two vertices in $\psi(H)$ there is a path of length at most $\tau$.
\end{proposition}

\begin{proof}
The statements about the maximum degree and number of vertices are trivial to check.
There are cycles of length $4$ spanning three parts due to the symmetry of $H$.
There are no triangles unless $\tau=3$.
By the self-duality of $H$, there is a $\tau$-path between every pair of vertices in the same part by traversing through exactly one full revolution of the cycle.
Moreover, by the parity of $\tau$ together with the self-duality of $H$, any two vertices in different parts are joined by a $\tau$-path possibly by changing directions several times along the cycle.
\end{proof}

\section{Summary and conclusion}\label{sec:conclusion}

Let us tie things together for Theorem~\ref{thm:constructions} before proposing further possibilities.

\begin{proof}[Proof of Theorem~\ref{thm:constructions}]
Table~\ref{tab:constructions} explicitly indicates which construction from Section~\ref{sec:circular} or~\ref{sec:cliques} is used in each lower bound for $\chi^t_g(d)$ listed in Theorem~\ref{thm:constructions}.
In the table, we have used the following notation.
The largest prime power not exceeding $d/2-1$ is denoted $q$, so that $2q+2 \sim d$ as $d\to\infty$. The inequalities in rows involving $q$ hold for all $d$ as $d\to\infty$ since the gap between two successive primes $p$ and $p'$ is $o(p)$~\cite{Ing37}. The mirror of a graph $G$ is denoted by $-{G}$.
We have written $({G_0}^{\alpha_0},\ldots,{G_{\lambda-1}}^{\alpha_{\lambda-1}})$ for the circular construction as described in Section~\ref{sec:circular}, with the adjacencies between $U^{(i)}$ and $U^{(i+1 \bmod \lambda)}$ defined according to $G_i$ along the $\alpha_i$-th coordinate. Bracketed factors $3$ and $5$ in the lefthand column require self-duality and are not necessarily valid for all values of $d$ as $d\to\infty$, as we describe below.

In the row for $t=4$ and girth $6$ we use a subgraph of the circular construction. It has vertex set $U^{(0)}\cup U^{(1)}\cup U^{(2)}$, with the edges between $U^{(0)}$ and $U^{(1)}$ embedded according to ${\cal P}_q$ along the zeroth coordinate (as in the circular construction), and edges between $U^{(1)}$ and $U^{(2)}$ embedded according to ${\cal P}_q$ along the first coordinate. There are no edges between $U^{(0)}$ and $U^{(2)}$. By the same arguments used for the circular construction, we conclude that the girth of the graph is $6$. Moreover, the distance properties of ${\cal P}_q$ ensure that $U^{(1)}$ induces a clique in the fourth power.

In the row for $t=4$ and girth $4$, we have an additional factor $2$, which is justified with the observation that $U^{(0)} \cup U^{(2)}$ induces a clique in the fourth power. This holds similarly for $U^{(0)} \cup U^{(1)}$ in the row for $t=7$.

For the third-to-last row, it is easily checked that, if $t=9$  or $t\ge 11$, then $t$ is expressible as a sum of at least three terms in $\{3,5\}$, so that the circular construction as per Theorem~\ref{thm:circularmain} need only be composed using ${\cal Q}_{q}$'s and ${\cal H}_{q}$'s. Then, as described at the end of Section~\ref{sec:circular}, we can unfold each coordinate into three copies, distributed evenly around the cycle, to achieve a girth $8$ construction. Optionally, if we use ${\cal Q}_{q_2}$ and ${\cal H}_{q_3}$, where $q_2$ is a power of $2$ and $q_3$ is a power of $3$ , then the use of self-dual embeddings ensures that we can freely change direction around the main cycle so that $U^{(0)} \cup U^{(\lambda)} \cup U^{(2\lambda)}$ induces a clique in the $t$-th power. By choosing $q_2$ and $q_3$ of similar magnitude (say, by using arbitrarily fine rational approximations of $\log_2 3$), we see that the factor $3$ improvement in the inequality holds for infinitely many $d$. This also explains the girth $6$ constructions for $t=6$ and $t=8$. A similar argument, where we instead unfold each coordinate into five copies, applies for the last two rows.
\end{proof}

\begin{table}
\centering
\begin{tabular}{| p{0.228\textwidth} | p{0.7\textwidth}|}
\hline
Bound & Construction that certifies the bound \\
\hline
$\chi^2_6(d) \gtrsim d^2$ & ${\cal P}_{q'}$, with $q'$ the largest prime power at most $d+1$ \\
\hline
$\chi^3_3(d) \gtrsim 3d^3/2^3$ & A ${\cal Q}_{q'}$-cycle of length $3$, with $q'$ the largest power of $2$ at most $d/2-1$ \\
\hline
$\chi^3_4(d) \gtrsim d^3/2^3$ & A matching contraction of ${\cal Q}_{q}$ \\
\hline
$\chi^4_4(d) \gtrsim 2d^4/2^4$ & $\left({{\cal P}_{q}}^0,-{{{\cal P}_{q}}^0},{{\cal P}_{q}}^1,-{{{\cal P}_{q}}^1}\right)$\\
\hline
$\chi^4_6(d) \gtrsim d^4/2^4$ & A ``non-circular'' $\left({{\cal P}_{q}}^0, {{\cal P}_{q}}^1\right)$ (see proof)\\
\hline
$\chi^5_4(d) \gtrsim 5d^5/2^5$ & An ${\cal H}_{q'}$-cycle of length $5$, with $q'$ the largest power of $3$ at most $d/2-1$ \\
\hline
$\chi^5_6(d) \gtrsim d^5/2^5$ & A matching contraction of ${\cal H}_{q}$ \\
\hline
$\chi^6_6(d) \gtrsim (3)d^6/2^6$ & $\left({{\cal Q}_{q}}^0,{{\cal Q}_{q}}^1\right)$, each coordinate unfolded into three copies \\
\hline
$\chi^7_6(d) \gtrsim 2d^7/2^7$ & $\left({{\cal Q}_{q}}^0,{{\cal P}_{q}}^1,-{{{\cal Q}_{q}}^0},-{{{\cal P}_{q}}^1},{{\cal P}_{q}}^2,{{\cal Q}_{q}}^0,-{{{\cal P}_{q}}^2} \right)$ \\
\hline
$\chi^8_6(d) \gtrsim (3)d^8/2^8 $ & $\left({{\cal Q}_{q}}^0,{{\cal H}_{q}}^1\right)$, each coordinate unfolded into three copies \\
\hline
$\chi^{t}_8(d) \gtrsim (3)d^t/2^t$, $t = 9$ or $t \ge 11$ & A circular construction with at least three $\tau_j$'s chosen from $\{3,5\}$ such that they sum to $t$, each coordinate unfolded into three copies\\
\hline
$\chi^{10}_6(d) \gtrsim (5)d^t/2^t$ &  A circular construction composed of two ${\cal H}_{q}$'s, 
each coordinate unfolded into five copies \\
\hline
$\chi^{t}_8(d) \gtrsim (5)d^t/2^t$, $t \ge 15$, $5 | t$ &  A circular construction composed only of ${\cal H}_{q}$'s, 
each coordinate unfolded into five copies \\
\hline
\end{tabular}
\caption{A list of constructions used in the proof of Theorem~\ref{thm:constructions}.\label{tab:constructions}}
\end{table}

Our work is a first systematic attempt at the problem of Alon and Mohar, although their conjecture --- which says for every positive $t$ there is a critical girth $g_t$ such that $\chi^t_{g_t}(d) = \Theta(d^t)$ and $\chi^t_{g_t+1}(d) = \Theta(d^t/\log d)$ --- remains wide open.
Because of the reliance upon incidence structures, it seems unlikely that our methods or similar ones could produce constructions of girth higher than $12$ or $16$.
We suspect though that $g_t$ exists and is linear in $t$. Irrespective of the existence or value of $g_t$, we conjecture the following in relation to Theorem~\ref{thm:cyclefree}.
\begin{conjecture}\label{conj:cyclefree}
The largest possible value of the chromatic number $\chi(G^t)$ of $G^t$, taken over all graphs $G$ of maximum degree at most $d$ containing as a subgraph no cycle of length $2t+2$ is $\Theta(d^t/\log d)$ as $d\to \infty$.
\end{conjecture}

In~\cite{KaKa14}, an edge-colouring version of Alon and Mohar's problem was proposed and studied.  This is related to a well-known conjecture of Erd\H{o}s and Ne\v{s}et\v{r}il, and to frequency allocation problems in ad hoc wireless networks. Techniques in the present paper carry over similarly, but we defer this to follow up work.

Determination of the extremal value of $\chi^t_3(d)$ up to a $1+o(1)$ factor as $d\to\infty$ is very enticing and is closely related to the degree diameter problem. 

Another interesting problem: what is the smallest possible value of the stability number $\alpha(G^t)$ of $G^t$, taken over all graphs $G$ of maximum degree at most $d$ and girth at least $g$?
To our knowledge, this natural extremal problem has not been extensively studied thus far.

\subsection*{Acknowledgements}

We are very grateful to the referees for their helpful comments and suggestions. We particularly thank one of the referees for pointing out Proposition~\ref{prop:cyclefree}.

%%%%%%%%%%%%%%%%%%%%%%%%%%%%%%%%%%%%%%%%%%%%%%%%%%%%%%%%%%%%%%%%%%%%%%

\bibliographystyle{abbrv}
\bibliography{distgirth}
\end{document}